\documentclass[11pt]{article}
\usepackage{amsthm, amsmath, amssymb, amsfonts, url, booktabs, tikz, setspace, fancyhdr, bm}
\usepackage{hyperref}
\usepackage{geometry}
\usepackage{hyperref, enumerate}
\usepackage[shortlabels]{enumitem}
\usepackage[babel]{microtype}
\usepackage[english]{babel}
\usepackage[capitalise]{cleveref}
\usepackage{comment}
\usepackage{bbm}
\usepackage{csquotes}
\usepackage{mathabx}
\usepackage{tikz}
\usepackage{graphicx}
\usepackage{float}
\usepackage{amsmath}

\counterwithin{figure}{section}

\newtheorem{theorem}{Theorem}[section]
\newtheorem{prop}[theorem]{Proposition}
\newtheorem{lemma}[theorem]{Lemma}

\newtheorem{claim}[theorem]{Claim}

\theoremstyle{definition}
\newtheorem{defn}[theorem]{Definition}
\newtheorem*{defn-non}{Definition}

\newlist{Case}{enumerate}{3}
\setlist[Case, 1]{%
    label           =   {\bfseries Case \arabic*.},
    labelindent=1em ,labelwidth=1cm, labelsep*=1em, leftmargin =!
}
\setlist[Case, 2]{%
    label           =   {\bfseries Subcase \arabic{Casei}.\arabic*.},
    labelindent=-1em ,labelwidth=1cm, labelsep*=1em, leftmargin =!
}
\setlist[Case, 3]{%
    label           =   {\bfseries Subsubcase \arabic{Casei}.\arabic{Caseii}.\arabic*.},
    labelindent=-1em ,labelwidth=1cm, labelsep*=1em, leftmargin =!
}

\newenvironment{poc}{\begin{proof}[Proof of claim]}{\end{proof}}

\newcommand{\vv}{\boldsymbol{v}}
\newcommand{\TT}{\boldsymbol{T}}
\newcommand{\RR}{\boldsymbol{R}}
\usepackage{todonotes}

\newcommand*{\abs}[1]{\lvert#1\rvert}

\title{The Frankl-Pach upper bound is not tight for any uniformity}

\author{
Gennian Ge\thanks{School of Mathematical Sciences, Capital Normal University, Beijing, China. Emails: gnge@zju.edu.cn, 3535935416@qq.com. Gennian Ge was supported by the National Key Research and Development Program of China under Grant 2020YFA0712100, the National Natural Science Foundation of China under Grant 12231014, and Beijing Scholars Program.}
\and 
Zixiang Xu\thanks{Extremal Combinatorics and Probability Group (ECOPRO), Institute for Basic Science (IBS), Daejeon, South Korea. Email: zixiangxu@ibs.re.kr. Supported by IBS-R029-C4.}
\and 
Chi Hoi Yip\thanks{School of Mathematics, Georgia Institute of Technology, Atlanta, United States. Email: cyip30@gatech.edu}
\and
Shengtong Zhang\thanks{Department of Mathematics, Stanford University, USA. E-mail address: stzh1555@stanford.edu}
\and
Xiaochen Zhao\footnotemark[1]
}

\begin{document}
\maketitle

\begin{abstract}
For any positive integers $n\ge d+1\ge 3$, what is the maximum size of a $(d+1)$-uniform set system in $[n]$ with VC-dimension at most $d$? In 1984, Frankl and Pach initiated the study of this fundamental problem and provided an upper bound $\binom{n}{d}$ via an elegant algebraic proof. Surprisingly, in 2007, Mubayi and Zhao showed that when $n$ is sufficiently large and $d$ is a prime power, the Frankl-Pach upper bound is not tight. They also remarked that their method requires $d$ to be a prime power, and asked for new ideas to improve the Frankl-Pach upper bound without extra assumptions on $n$ and $d$.  

In this paper, we provide an improvement for any $d\ge 2$ and $n\ge 2d+2$, which demonstrates that the long-standing Frankl-Pach upper bound $\binom{n}{d}$ is not tight for any uniformity. Our proof combines a simple yet powerful polynomial method and structural analysis.
\end{abstract}

\section{Introduction}
The \emph{Vapnik-Chervonenkis dimension} (\emph{VC-dimension}) is a cornerstone concept in statistical learning and machine learning. It quantifies the capacity of a set of functions and assesses a model's potential to generalize beyond its training data. Widely applied in computational learning theory, artificial intelligence, and statistical inference, the VC-dimension plays a pivotal role in analyzing classification algorithms and evaluating the performance of machine learning models. Its formal definition is as follows.

\begin{defn}\label{def:VCdim}
    Let $\mathcal{F} \subseteq 2^V$ be a set system on the ground set $V$. The VC-dimension of $\mathcal{F}$ is the largest size of subset $S\subseteq V$ such that for every subset $S'\subseteq S$, there exists a member $F_{S'}\in\mathcal{F}$ such that $F_{S'}\cap S=S'$.
    \end{defn}
A natural question that arises is: What is the maximum size of a set system with a given VC-dimension? This fundamental question in the theory of VC-dimension was independently resolved in the 1970s by three different groups~\cite{1972JCTASauer,1972PACJMShelah,1971TPAVCVC}.

\begin{lemma}\label{lem:SSLemma}
    Let $\mathcal{F}\subseteq 2^{V}$ be a set system with VC-dimension at most $d$, then
$|\mathcal{F}|\le\sum\limits_{i=0}^{d}\binom{|V|}{i}$.
\end{lemma}
Furthermore, the upper bound in~\cref{lem:SSLemma} is sharp, as shown by considering a Hamming ball with radius $d$. \cref{lem:SSLemma} also stands as an important result in extremal combinatorics. In particular, the connection between VC-dimension and extremal combinatorics has sparked significant excitement in recent years; see for example~\cite{2024BVCSunflower,2019DCGBVCEH,2021ErdosSchur,2023BVCSunflower,2024GraphToGeom,2023BVCErdosHajnal,2023SukMatchingLemma}.

While~\cref{lem:SSLemma} provides a tight upper bound for non-uniform set systems with a given VC-dimension, it is natural to investigate the maximum size of a uniform set system with small VC-dimension. Indeed, in the 1980s, Frankl and Pach~\cite{1984Franklpach} initiated the study of this fundamental problem and provided a clean upper bound with an elegant proof.

\begin{theorem}[\cite{1984Franklpach}]\label{thm:FranklPach}
   Let $n,d$ be positive integers with $n\ge d+1$. If $\mathcal{F}\subseteq\binom{[n]}{d+1}$ is a set system with VC-dimension at most $d$, then $|\mathcal{F}|\le\binom{n}{d}$.
\end{theorem}
Frankl and Pach~\cite{1984Franklpach} as well Erd\H{o}s~\cite{1984Erdos} further conjectured that for sufficiently large $n$, the maximum value is $\binom{n-1}{d}$ achieved when $\mathcal{F}$ forms a star. However, two unexpected results emerged over the past four decades:
   \begin{enumerate}
       \item[\textup{(1)}] In 1997, Ahlswede and Khachatrian~\cite{1997CombFan} disproved the above conjecture by constructing such a set system of size $\binom{n-1}{d}+\binom{n-4}{d-2}$ when $n\ge 2(d+1)$.
       \item[\textup{(2)}] In 2007, Mubayi and Zhao~\cite{2007JAC} cleverly combined the high-order inclusion matrix and the sunflower lemma to show that, when $d=p^{t}$ for some prime $p$ and positive integer $t$ and $n$ is sufficiently large relative to $d$, the size of such a set system is at most $\binom{n}{d}-\log_{p}{n}+C_{d}$, where $C_{d}$ is related to the sunflower lemma~\cite{2021Annals,1960Sunflower}. They also remarked~\cite[Proposition 10]{2007JAC} that their method requires $d$ to be a prime power, and claimed that improving the Frankl-Pach upper bound for other values of $d$ will most likely require some new ideas.
   \end{enumerate}

In this paper, we show that the upper bound in~\cref{thm:FranklPach} can be improved for all positive integers $n,d$ with $d\ge 2$ and $n\geq 2d+2$. 
\begin{theorem}\label{thm:Main}
    Let $n,d$ be positive integers with $d\ge 2$ and $n\ge 2d+ 2$. If $\mathcal{F}\subseteq\binom{[n]}{d+1}$ is a set system with VC-dimension at most $d$, then $|\mathcal{F}|<\binom{n}{d}$. 
\end{theorem}
In particular the condition $n\ge 2d+2$ is optimal since when $n=2d+1$, the upper bound $\binom{n}{d}$ is achieved by $\mathcal{F} = \binom{[n]}{d + 1}$, which was remarked by Frankl and Pach~\cite{1984Franklpach}. Our proof combines a simple yet powerful polynomial method with clean structural analysis. We begin with a concise proof of~\cref{thm:FranklPach} in~\cref{section:WarmUp} as an introductory step, which we then refine to establish~\cref{thm:Main} in~\cref{section:Main}.\\

{\bf \noindent Notations.} We usually regard $[n]=\{1,2,\ldots,n\}$ as the ground set. For a subset $A\subseteq [n]$, we use $\binom{A}{k}$ to denote the family of all subsets of $A$ with size $k$ and $\binom{A}{\leq k}$ to denote the family of all subsets of $A$ with size at most $k$; we also use $\boldsymbol{v}_A \in \{0,1\}^n$ to denote the characteristic vector of $A$. For a set system $\mathcal{F}$, we denote the \emph{$k$-shadow} of $\mathcal{F}$ as $\partial_{k}{\mathcal{F}}:=\{T\in\binom{[n]}{k}:T\subseteq F\ \text{for\ some\ }F\in\mathcal{F}\}$. We write $\boldsymbol{x}=(x_1,x_2, \ldots, x_n)$ and use $V_{n,d}$ to denote the vector space spanned by squarefree monomials in $\mathbb{R}[\boldsymbol{x}]$ with degree at most $d$.

\section{Warm-up: Frankl-Pach upper bound via polynomial method}\label{section:WarmUp}
In this section, we provide an alternative proof of~\cref{thm:FranklPach} using the multi-linear polynomial method. The following triangular criterion is useful for proving that a sequence of polynomials is linearly independent. Since this result is standard, we omit the proof here.

\begin{lemma}\label{lem:triangular}
    Let $f_{1},f_{2},\ldots,f_{m} \in \mathbb{R}[\boldsymbol{x}]$. If $\boldsymbol{v}^{(1)},\boldsymbol{v}^{(2)},\ldots,\boldsymbol{v}^{(m)}$ are vectors in $\mathbb{R}^n$ such that $f_{i}(\boldsymbol{v}^{(i)})\neq 0$ for $1\le i\le m$ and $f_{i}(\boldsymbol{v}^{(j)})=0$ for $i>j$, then $f_{1},f_{2},\ldots,f_{m}$ are linearly independent.
\end{lemma} 
 We then build the following relationship between the size of $\mathcal{F}$ and its $d$-shadow.
\begin{theorem}\label{prop:Shadow}
    Let $\mathcal{F}\subseteq\binom{[n]}{d+1}$ be a set system with VC-dimension at most $d$, then $|\mathcal{F}|\le |\partial_{d}\mathcal{F}|$.
\end{theorem}
We remark that \cref{prop:Shadow} is of independent interest. For instance, it extends the celebrated Katona's shadow theorem~\cite{1964Katona}, which asserts that any $(d+1)$-uniform intersecting family $\mathcal{F}$ satisfies $|\mathcal{F}| \leq |\partial_d \mathcal{F}|$; this follows from the fact that any such intersecting family has VC-dimension at most $d$. Moreover, since $|\partial_{d}\mathcal{F}|\le\binom{n}{d}$, we can see~\cref{prop:Shadow} implies~\cref{thm:FranklPach}.

\begin{proof}[Proof of~\cref{prop:Shadow}]
Let $\mathcal{F} = \{F_1, F_2, \ldots, F_m\} \subseteq \binom{[n]}{d+1}$ be a set system with VC-dimension at most $d$. We now consider the following set systems and their corresponding polynomials.

\begin{enumerate}
    \item Since the VC-dimension of $\mathcal{F}$ is at most $d$, for each set $F_{i}\in\mathcal{F}$, there exists some $B_{i}$ such that $F_{i}\cap F\neq B_{i}$ for all $F\in\mathcal{F}$. For each $F_{i}\in \mathcal{F}$, let $f_{F_{i}}(\boldsymbol{x})=\prod\limits_{j\in B_{i}}x_{j}\cdot\prod\limits_{j\in F_{i}\setminus B_{i}}(x_{j}-1)-\prod\limits_{j\in F_{i}}x_{j}$.
    \item Let $\mathcal{H}=\binom{[n]}{\le d-1}$. For each set $H\in\mathcal{H}$, define $h_{H}(\boldsymbol{x})=\sum\limits_{j\notin H}x_{j}\cdot\prod\limits_{j\in H}x_{j}+(|H|-d-1)\prod\limits_{j\in H}x_{j}$.
    
    \item Let $\mathcal{G}=\binom{[n]}{d}\setminus\partial_{d}\mathcal{F}$. For each set $G\in \mathcal{G}$, define $g_{G}(\boldsymbol{x})=\prod\limits_{j\in G}x_{j}$.
\end{enumerate}

\begin{claim}\label{claim:LID}
  The polynomials $\{f_{F_{i}}\}_{F_{i}\in\mathcal{F}}$, $\{h_{H}\}_{H\in\mathcal{H}}$, and  $\{g_{G}\}_{G\in\mathcal{G}}$ are linearly independent.
\end{claim}
\begin{poc}
We first order the sets linearly in the order: $\mathcal{F} < \mathcal{H} < \mathcal{G}$, and then put the elements of $\mathcal{F}$ and $\mathcal{G}$ in arbitrary order and the elements of $\mathcal{H}$ in non-decreasing order.
We shall show that these polynomials satisfy the triangular criterion in Lemma~\ref{lem:triangular} by evaluating them on the characteristic vectors $\boldsymbol{v}_{A}$ for $A \in \mathcal{F} \cup \mathcal{H} \cup \mathcal{G}$.
We consider the following six cases:
\begin{itemize}
    \item For distinct $F_{i},F_{j}\in\mathcal{F}$ with $i>j$, first, it is evident that $f_{F_{i}}(\boldsymbol{v}_{F_{i}})=-1$ since $B_{i}\neq F_{i}$. Moreover, since $F_{j}\cap F_{i}\neq B_{i}$, we can see $f_{F_{i}}(\boldsymbol{v}_{F_{j}})=0$.
    
 \item For any $F\in\mathcal{F}$ and $H\in\mathcal{H}$, we have $h_{H}(\boldsymbol{v}_{F})=0$ since $|F|=d+1$.
    
    \item For any $F\in\mathcal{F}$ and $G\in\mathcal{G}$, we have $g_{G}(\boldsymbol{v}_{F})=0$ since $G$ does not belong to $\partial_{d}\mathcal{F}$.
   
    \item For the sets in the family $\mathcal{H}$, we sort them by size in non-decreasing order, placing smaller sets before larger ones. That is, for distinct $H_i, H_j \in \mathcal{H}$, we have $|H_j| \le |H_i|$ if $j \le i$. First, it is easy to check that $h_{H_i}(\boldsymbol{v}_{H_i}) \neq 0$ since $0 \le |H_i| \le d-1$. Furthermore, we observe that when $i > j$, $h_{H_i}(\boldsymbol{v}_{H_j}) = 0$, because $|H_j| \le |H_i|$ implies that $H_i \not\subseteq H_j$.
    \item For any $H\in\mathcal{H}$ and $G\in\mathcal{G}$, obviously we have $g_{G}(\boldsymbol{v}_{H})=0$ since $|H|\le d-1<d$.
    \item For any distinct $G_{i},G_{j}\in\mathcal{G}$, we have $g_{G_{i}}(\boldsymbol{v}_{G_{i}})=1$ and $g_{G_{i}}(\boldsymbol{v}_{G_j})=0$ since $|G_{i}|=|G_{j}|$ implies that $G_{i}\not\subseteq G_{j}$.
\end{itemize}
This finishes the proof.
\end{poc}
By~\cref{claim:LID}, we have $|\mathcal{F}|+\sum\limits_{i=0}^{d-1}\binom{n}{i}+\big(\binom{n}{d}-|\partial_{d}\mathcal{F}|\big)\le\sum\limits_{i=0}^{d}\binom{n}{i}$ since these polynomials are in the vector space $V_{n,d}$, which has dimension $\sum\limits_{i=0}^{d}\binom{n}{i}$. This yields $|\mathcal{F}|\le|\partial_{d}\mathcal{F}|$, as desired.
\end{proof}

\section{Proof of~\cref{thm:Main}}\label{section:Main}
In this section, we continue using the same notation as in the proof of~\cref{prop:Shadow}, in particular the sets $B_{i}$ and the family $\mathcal{H}$.
 
To prove~\cref{thm:Main}, we assume for contradiction that $|\mathcal{F}|=\binom{n}{d}$, and thus $\left|\partial_{d} \mathcal{F}\right|=\binom{n}{d}$ as well by~\cref{prop:Shadow}. Our proof then consists of two parts, we first introduce an additional polynomial in $V_{n,d}$, hoping that this new polynomial is linearly independent from the polynomials $\{f_{F_i}\}_{F_{i}\in\mathcal{F}}$, $\{h_{H}\}_{H\in\mathcal{H}}$ and $\{g_{G}\}_{G\in\mathcal{G}}$ (under our assumption, $\mathcal{G}=\emptyset$) in the proof of~\cref{prop:Shadow}, thereby achieving an improvement. Otherwise, we can derive a desirable combinatorial property~(see~\cref{claim:ExactOne}) about this set system. However, based on an application of the classical Kruskal-Katona theorem, such a combinatorial property cannot hold.
\subsection{Adding an extra polynomial}
 For each set $Y\in \binom{[n]}{d+1} \setminus \mathcal{F}$, and for each proper subset $Z\subsetneq Y$, we define the following polynomial
\begin{equation*}
    y_{Y, Z}(\boldsymbol{x})=\prod\limits_{j\in Z}x_{j}\cdot\prod\limits_{j\in Y\setminus Z}(x_{j}-1) - \prod\limits_{j\in Y}x_{j}.
\end{equation*}
Observe that $y_{Y,Z}$ is multilinear and has degree at most $d$, thus $y_{Y,Z}$ lies in the vector space $V_{n,d}$.

\begin{claim}\label{claim:ExactOne}
    If $|\mathcal{F}| = \binom{n}{d}$, then for each $Y \in \binom{[n]}{d+1}\setminus\mathcal{F}$ and each $Z \subsetneq Y$, there is exactly one $i$ such that $F_i \cap Y = Z = B_i$.
\end{claim}
\begin{poc}
    Suppose that there exists some pair $(Y,Z)$ such that the number of $F_{i}\in\mathcal{F}$ with $F_{i}\cap Y=Z=B_{i}$ is $m_{0}\neq 1$. We then claim that $y_{Y, Z}$, $\{f_{F_{i}}\}_{F_{i}\in\mathcal{F}}$ and $\{h_{H}\}_{H\in\mathcal{H}}$ are linearly independent, which implies that $|\mathcal{F}|\le\binom{n}{d}-1$, a contradiction. To see this, note that $\mathcal{G}=\emptyset$, and suppose there exist coefficients (that are not all zero) $\alpha$, $\boldsymbol{\beta}=(\beta_{1},\beta_{2},\ldots,\beta_{m})$, and $\boldsymbol{\gamma}=(\gamma_{1},\gamma_{2},\ldots,\gamma_{w})$, where $w=|\mathcal{H}|=\sum\limits_{i=0}^{d-1}\binom{n}{i}$, such that
    \begin{equation*}
        \alpha\cdot y_{Y,Z}+ \sum\limits_{i\in [m]}\beta_{i}f_{F_{i}}+\sum\limits_{j\in [w]}\gamma_{j}h_{H_{j}}=0.
    \end{equation*}
    Note that $y_{Y,Z}(\boldsymbol{v}_{Y})=-1$ and $y_{Y,Z}(\boldsymbol{v}_{F_{i}})$ equals $(-1)^{d+1-|Z|}$ if $Y\cap F_{i}=Z$ and equals $0$ otherwise. $f_{F_{i}}(\boldsymbol{v}_{Y})$ equals $(-1)^{d+1-|B_{i}|}$ if $Y\cap F_{i}=B_{i}$ and equals $0$ otherwise. Mirroring the previous argument in the proof of~\cref{claim:LID}, we can obtain the following matrix equality by evaluating the left-hand side at the indicator vectors $\boldsymbol{v}_{A}$ for $A \in \{Y\} \cup \mathcal{F} \cup \mathcal{H}$:
    \begin{equation*}
\left(\begin{matrix}
-1 &\TT^\top & \boldsymbol{0} \\
\RR & -\boldsymbol{I}_{m} & \boldsymbol{0} \\
 * & * & \boldsymbol{A}
\end{matrix}\right)
\begin{bmatrix} \alpha \\ \boldsymbol{\beta} \\  \boldsymbol{\gamma} \\ \end{bmatrix} 
= \boldsymbol{0},
\end{equation*}
where the vectors $\TT=(T_{1},\ldots,T_{m})^\top$ and $\RR=(R_{1},\ldots,R_{m})^\top$ are given by
$$T_i = f_{F_i}(\vv_{Y}) = \begin{cases}
    (-1)^{d + 1 - \abs{B_i}}, &\ \textup{if}\ Y \cap F_i = B_i  \\
    0, &\ \text{otherwise}
\end{cases},$$
$$R_i = y_{Y, Z}(\vv_{F_i}) = \begin{cases}
    (-1)^{d + 1 - \abs{Z}},&\ \textup{if}\ Y \cap F_i = Z \\
    0, &\ \text{otherwise}
\end{cases},$$
$\boldsymbol{I}_m$ is the $m \times m$ identity matrix, $\boldsymbol{A}$ is a $w\times w$ lower triangular matrix with nonzero diagonal entries, and we use $*$ to denote the sub-matrices whose exact values are irrelevant to our analysis. 

Denote the coefficient matrix by \(\boldsymbol{D}\). The matrix \(\boldsymbol{D}\) must be singular; otherwise, the system would have only the trivial solution for \((\alpha, \boldsymbol{\beta}, \boldsymbol{\gamma})\). On the other hand, using the elementary row operations, we can transform $\boldsymbol{D}$ into the matrix
\begin{equation*}
\left(\begin{matrix}
-1 + \TT^\top \RR & \boldsymbol{0} & \boldsymbol{0} \\
\RR & -\boldsymbol{I}_{m} & \boldsymbol{0} \\
 * & * & \boldsymbol{A}
\end{matrix}\right).
\end{equation*}
We can directly compute that
$$\TT^\top \RR = \sum_{i = 1}^{m} T_i R_i= m_0.$$
As $m_0 \neq 1$, the transformed matrix is lower triangular with nonzero diagonal entries, yielding that the coefficient matrix $\boldsymbol{D}$ is also invertible, a contradiction. This finishes the proof.
\end{poc}

\newcommand{\cF}{\mathcal{F}}
\newcommand{\cG}{\mathcal{G}}
\subsection{Impossibility via the Kruskal-Katona theorem}
We now show that the property given in~\cref{claim:ExactOne} can never be satisfied when $n \geq 2d + 2$. We shall make use of the Kruskal-Katona theorem~\cite{1966Katona,1963Kruskal}. More precisely, we take advantage of the following useful form due to Lov\'{a}sz~\cite [Exercise. 13.31(b)]{1979KKLovasz}. When $\alpha \in \mathbb{R}^+$, let $\binom{\alpha}{k}:=\frac{\alpha(\alpha-1)\cdots(\alpha-k+1)}{(k)!}$.
\begin{theorem}[Kruskal-Katona theorem]\label{thm:KK}
    For any $\alpha \in \mathbb{R}^+$, if $\cG$ is a $(d + 1)$-uniform set family with $\abs{\cG} = \binom{\alpha}{d + 1}$, then 
    $$\abs{\partial_{d}\cG} \geq \binom{\alpha}{d} = \abs{\cG} \cdot \frac{d + 1}{\alpha - d}.$$
\end{theorem}

More precisely, we can show the following result, which implies~\cref{thm:Main}.
\begin{prop}\label{Prop:Impossible}
    Let $n \geq 2d + 2$. Then there is no $\cF = \{F_1, \ldots, F_m\} \subseteq \binom{[n]}{d + 1}$ together with $B_i \subseteq F_i$ and $\abs{\cF} = \binom{n}{d}$ satisfying the following properties simultaneously:
    \begin{enumerate}
        \item[\textup{(1)}] For each $Y \in \binom{[n]}{d + 1} \setminus \cF$ and each $Z \subsetneq Y$, there exists exactly one $F_i \in \cF$ with $F_i \cap Y = Z = B_i$.
        \item[\textup{(2)}] $F_i \cap F_j \neq B_i$ for any $1\leq i,j\leq m$.
    \end{enumerate}
\end{prop}
\begin{proof}[Proof of~\cref{Prop:Impossible}]

We consider the set system $\mathcal{Y} = \binom{[n]}{d + 1} \setminus \cF$ with $\abs{\mathcal{Y}} = \binom{n}{d + 1} - \binom{n}{d}$. {Since $n\geq 2d+2$, $\mathcal{Y}$ is nonempty.} 
By property (1), for each $Y \in\mathcal{Y}$ and each $Z \subseteq Y$ of size $d$, there is exactly one $i\in [m]$ such that $F_i \cap Y=Z = B_i$. {This shows that $\partial_{d}\mathcal{Y} \subseteq \{B_i: \abs{B_i} = d\}.$} On the other hand, for each $B_{i}$ of size $d$, there are $n-d-1$ many $Y\in\mathcal{Y}$ containing $B_{i}$ {and in particular $B_i \in \partial_{d} \mathcal{Y}$. Indeed, if $|B_i|=d$, then $B_i$ is a subset of $n-d$ sets in $\binom{[n]}{d+1}$, and among the sets in $\mathcal{F}$, only $F_i$ contains $B_i$: if $B_i \subseteq F_j$ for some $j\neq i$, then we would have $F_i \cap F_j=B_i$.}

Thus, we have shown that
$$\partial_{d}\mathcal{Y} = \{B_i: \abs{B_i} = d\}.$$ Observe that if $i \neq j$ and $\abs{B_i}=\abs{B_j}=d$, then $B_i\neq B_j$, for otherwise $F_i \cap F_j=B_i$, violating the property (2) of $\cF$. This shows that $\abs{\partial_{d}\mathcal{Y}}=|\{1\leq i \leq m: \abs{B_i}=d\}|$.

 Double counting the number of pairs $(Y, i)$ with $Y \cap F_i = B_i$ and $\abs{B_i} = d$, we can obtain
$$\abs{\partial_{d}\mathcal{Y}} \cdot (n - d - 1) = (d + 1) \cdot \abs{\mathcal{Y}},$$
which yields that
$$\abs{\partial_{d}\mathcal{Y}} = \frac{d + 1}{n - d - 1}\cdot \abs{\mathcal{Y}}.$$
Let $\alpha\in\mathbb{R}^{+}$ be the positive real number satisfying $\abs{\mathcal{Y}} = \binom{\alpha}{d + 1}$. \cref{thm:KK} then implies that
$$\frac{d + 1}{n - d - 1} = \frac{\abs{\partial_{d}\mathcal{Y}}}{\abs{\mathcal{Y}}} \geq \frac{d + 1}{\alpha - d},$$
which gives $\alpha \geq n - 1$. Therefore we have
$$\binom{n}{d + 1} - \binom{n}{d} = \abs{\mathcal{Y}} \geq \binom{n - 1}{d + 1}$$
which is impossible by direct computation, since
$$\binom{n}{d + 1} - \binom{n}{d} - \binom{n - 1}{d + 1} = \binom{n - 1}{d} - \binom{n}{d} < 0.$$
This finishes the proof.
\end{proof}
\section{Concluding remarks}
In this paper, we focus on a fundamental yet unresolved question: what is the maximum size of a $(d+1)$-uniform set system if its VC-dimension is bounded above by $d$? The main contribution of this work is to show that the longstanding Frankl-Pach upper bound $\binom{n}{d}$ is not tight for any uniformity. Meanwhile, we provide a simple polynomial method that establishes a direct connection between the set system and its shadow, which generalizes Katona's shadow theorem~\cite{1964Katona}. 

In general, Frankl and Pach~\cite{1984Franklpach} showed that for any $n\ge r\ge d+1$, the size of any $r$-uniform set system on $[n]$ with VC-dimension at most $d$ can be upper bounded by $\binom{n}{d}$. It would be interesting to further improve this upper bound. 

\section*{Acknowledgment}
The authors would like to express our gratitude to the anonymous reviewers for the detailed and constructive comments which are very helpful to the improvement of the technical presentation of this paper. 
    
\bibliographystyle{abbrv}
\bibliography{FranklPach}
\end{document}